\documentclass[10pt]{article} 
\usepackage[utf8]{inputenc}
\usepackage[top=30mm, left=30mm, right=30mm, bottom=30mm]{geometry}
\usepackage{amsmath,amssymb,amsthm}
\usepackage{dsfont}
\usepackage{color}
\usepackage{lineno,hyperref}
\usepackage{yfonts}
\usepackage{appendix}
\usepackage{babel} 

\newtheorem{theorem}{Theorem}[section]
\newtheorem{lem}[theorem]{Lemma}

\newtheorem{prop}[theorem]{Proposition}
\newtheorem{rem}[theorem]{Remark}
\theoremstyle{definition}
\newtheorem{dfn}[theorem]{Definition}

\numberwithin{equation}{section}

\newcommand{\del}{\partial}

\DeclareMathOperator*{\divv}{div}

\DeclareMathOperator*{\Id}{Id}
\DeclareMathOperator*{\loc}{loc}

\newcommand*{\Bscr}{\mathcal B}

\newcommand*{\Dscr}{\mathcal D}

\newcommand*{\Fscr}{\mathcal F}

\newcommand*{\Mscr}{\mathcal M}

\newcommand*{\Pscr}{\mathcal P}

\newcommand*{\N}{\mathbb{N}}

\newcommand*{\R}{\mathbb{R}}

\newcommand*{\vrho}{\varrho}
\newcommand*{\divvrho}{\textup{div}_\vrho}

\DeclareMathOperator*{\converge-n}{\xrightarrow{\textit{n}\to \infty}}

\definecolor{orange}{rgb}{1.0, 0.55, 0.0} 

\begin{document}
\title{Weighted $L^1$-semigroup approach for nonlinear Fokker--Planck equations and generalized Ornstein--Uhlenbeck processes}

\author{Marco Rehmeier\footnote{Faculty of Sciences, Scuola Normale Superiore Pisa, Italy. E-Mail: mrehmeier@math.uni-bielefeld.de }}
\date{\today}
\maketitle
\begin{abstract}
For the nonlinear Fokker--Planck equation
$$	\partial_tu = \Delta\beta(u)-\nabla \Phi \cdot \nabla \beta(u) - \textup{div}_{\varrho}\big(D(x)b(u)u\big),\quad (t,x) \in (0,\infty)\times \R^d,$$
where $\varrho = \exp(-\Phi)$ is the density of a finite Borel measure and $\nabla \Phi$ is unbounded, we construct mild solutions with bounded initial data via the Crandall--Liggett semigroup approach in the weighted space $L^1(\mathbb{R}^d,\mathbb{R};\varrho dx)$. By the superposition principle, we lift these solutions to weak solutions to the corresponding McKean--Vlasov SDE, which can be considered a model for generalized nonlinear perturbed Ornstein--Uhlenbeck processes. Finally, for these solutions we prove the nonlinear Markov property in the sense of McKean.
\end{abstract}

\section{Introduction}
We study the nonlinear parabolic Fokker--Planck equation on $\R^d$, $d \geq 1$,
\begin{equation}\label{equation-main}
	\del_tu = \Delta\beta(u)-\nabla \Phi \cdot \nabla \beta(u) - \divvrho\big(D(x)b(u)u\big),\quad (t,x) \in (0,\infty)\times \R^d,
\end{equation}
where $\vrho: \R^d \to \R$ is the density of a finite Borel measure, $\beta:\R \to \R $ is increasing, $D: \R^d \to \R^d$ and $b: \R \to \R$ are bounded, $\nabla \Phi = \nabla \log \vrho$, and $\divvrho$ denotes the dual operator to the gradient in $L^2(\R^d,\R;\vrho dx)$ (please see Hypothesis 1 in Section \ref{sect2} for the precise assumptions). Our aim is twofold: First, we apply the well-known Crandall-Liggett semigroup approach in the weighted space $L^1(\R^d,\R;\vrho dx)$ to construct mild solutions $u \in C\big([0,\infty),L^1(\R^d,\R;\vrho dx)\big)$  to \eqref{equation-main} (see Definition \ref{def:gen-sol}). Secondly, we show that these solutions, when started from probability density initial data, are the curves of $1D$-time marginals of probabilistically weak solutions to the associated nonlinear McKean--Vlasov stochastic differential equation (SDE)
\begin{align}\label{eq:SDE}
	dX_t &= \bigg[D(X_t)b\big(v(t,X_t)\vrho^{-1}(X_t)\big) - \frac{\beta(v(t,X_t)\vrho^{-1}(X_t))}{v(t,X_t)\vrho^{-1}(X_t)}\nabla \Phi(X_t)\bigg]dt +\sqrt{2 \frac{\beta(v(t,X_t)\vrho^{-1}(X_t))}{v(t,X_t)\vrho^{-1}(X_t)}}dB_t,
	\\ \mathcal{L}(X_t) &= v(t,x)dx \notag.
\end{align}
Here $\mathcal{L}(X)$ denotes the distribution of a random variable $X$, and $B$ is an $\R^d$-Brownian motion. We also show that these solutions to \eqref{eq:SDE} constitute a nonlinear Markov process in the sense of McKean, see \cite{McKean1-classical,R./Rckner_NL-Markov22}.

\textbf{Related results.}
Existence, uniqueness, asymptotic behavior and probabilistic representation of solutions to nonlinear Fokker--Planck equations, in particular of type
\begin{equation}\label{eqeqeq}
	\partial_t u = \Delta \beta(u)-\divv\big(B(x,u)u\big),
\end{equation}
where $B: \R^{d+1} \to \R$ is bounded,
have frequently been investigated in recent years, see for instance  \cite{BR18_2,NLFPK-DDSDE5,BR22,BR-IndianaJ} and the references therein. In \cite{NLFPK-DDSDE5}, for $B(x,r) = D(x)b(r)$ and under appropriate assumptions on the coefficients, the Crandall--Liggett semigroup approach in $L^1 = L^1(\R^d,\R;dx)$ is applied to construct an $L^1$-contraction semigroup $S(t), t\geq 0$, of mild solutions $t\mapsto S(t)(u_0)$, where the initial data $u_0$ are from $L^1$. In fact, in the same paper, this result is even extended to finite Radon measures as initial data, and an $L^1-L^\infty$-regularization result is obtained, i.e. solutions, even those equal to degenerate measures at $t=0$, are $L^1 \cap L^\infty$-valued for $t>0$. Since these solutions preserve non-negativity and total mass, starting from probability initial data yields solution curves of probability measure densities. By the Ambrosio--Figalli--Trevisan superposition principle (see \cite{Trevisan16},\cite{BR18_2}), such probability solutions are the $1D$-time marginals of solutions to the McKean--Vlasov SDE
\begin{equation*}
	dX_t = D(X_t)b(u(t,X_t))dt + \sqrt{2 \frac{\beta(u(t,X_t))}{u(t,X_t)}}dB_t.
\end{equation*}
Solutions to the latter equation, at least if $D$ is a gradient vector field, are called nonlinear distorted Brownian motion, see \cite{RRW20}. 
We would also like to mention further recent existence result obtained via the $L^1$-semigroup approach, for instance \cite{BR22-frac,BRS23}, where the authors solve equations of type \eqref{eqeqeq}, but $\Delta$ is replaced by a nonlocal operator, e.g. a fractional Laplacian. 

\textbf{Weighted semigroup approach.}
In the present paper, we initiate a similar program for \eqref{equation-main}, but, in contrast to all previously mentioned works, in the weighted space $L^1(\R,\R;\vrho dx)$. Our main result can be summarized as follows (please see Theorem \ref{main-thm} for the precise statement): If Hypothesis 1 below is satisfied, for any $u_0 \in L^1(\R^d,\R;\vrho dx)\cap L^\infty$ there is a mild solution $S(t)u_0, t \geq 0,$ to \eqref{equation-main}, and $(S(t))_{t\geq 0}$ is a semigroup of $L^1(\R^d,\R;\vrho dx)$-contractions. The proof is inspired by the one in \cite{NLFPK-DDSDE5}. One advantage of our approach is that, compared to the standard approach, it allows to consider nonlinear unbounded first-order perturbations, which are not of divergence--type. Indeed, a main example for $\nabla \Phi$ is $\Phi(x) = \frac{-|x|^2}{2}$, $x \in \R^d$. Equation \eqref{equation-main} can be treated by the semigroup approach in $L^1(\R^d,\R;\vrho dx)$, since the operator $Lf:= \Delta f - \nabla \Phi \cdot \nabla f$ is symmetric on $L^2(\R^d,\R;\vrho)$. In fact, $\vrho dx$ is an infinitesimally invariant measure for $L$, i.e. $L^*\vrho dx = 0$, where $L^*$ denotes the formal dual operator of $L$. The pair $(L,\vrho dx)$ can be considered a generalization of $(\Delta, dx)$. In this general setting we lose, however, very nice mapping properties of the fundamental solution to $\Delta$ in Marcinkiewicz spaces. Hence, so far we are only able to prove mild existence for initial data $u_0$ from $L^1(\R^d,\R;\vrho dx)$ which are also bounded. We can remove the boundedness restriction on $u_0$ when $\beta$ is globally Lipschitz, see Proposition \ref{prop:prop1}. The case of more general initial data, e.g. degenerate finite Borel measures, is not addressed in the present paper.

\textbf{Generalized Ornstein--Uhlenbeck processes.}
As said before, a second aspect of this work is the probabilistic representation of solutions to \eqref{equation-main}. Since our mild solutions $t\mapsto S(t)u_0$ are also distributional solutions (see Def.\eqref{def:distr-sol}), by the aforementioned superposition principle, we obtain weak solutions to \eqref{eq:SDE} with $1D$-time marginals $(S(t)u_0)_{t\geq 0}$. Solutions to \eqref{eq:SDE} can be considered \textit{generalized nonlinear perturbed Ornstein--Uhlenbeck processes}, since for $D = 0$, $\beta(r) = \frac{\sigma^2}{2}r$ and $\Phi(x) = -\frac{|x|^2}{2}$, the SDE becomes the classical Ornstein--Uhlenbeck equation. We stress that both \eqref{equation-main} and \eqref{eq:SDE} have nonlinear, irregular coefficients (of so-called Nemytskii-type, i.e. the dependence of the coefficients on a measure is pointwise via its density w.r.t. Lebesgue measure). 

Finally, we prove that the laws of the solutions to \eqref{eq:SDE} with $1D$-time marginals $(S(t)u_0)_{t\geq 0}$ constitute a \textit{nonlinear Markov process}. Such processes were suggested by McKean in 1966 \cite{McKean1-classical} and were recently revisited in \cite{R./Rckner_NL-Markov22}. McKean's vision was to represent solutions to nonlinear PDEs as the $1D$-time marginals of stochastic processes satisfying a nonlinear Markov property. Our results in Section \ref{sect:stoch-section} are a contribution to his vision for equation \eqref{equation-main}, which to the best of our knowledge has not been considered in this direction before.

\textbf{Organization.}
The paper is organized as follows. In Section \ref{sect2}, we state our assumptions on the coefficients, recall the notion of mild and distributional solutions to \eqref{equation-main} and state the new existence results, Theorem \ref{main-thm} and Proposition \ref{prop:prop1}. Moreover, we state the key lemma towards these results, Lemma \ref{main-lemma}. Section \ref{sect:proof-main-lemma} is devoted to the proof of this lemma. In Section \ref{sect:stoch-section}, we lift our solutions to \eqref{equation-main} to solutions to \eqref{eq:SDE} and prove the nonlinear Markov property.
\\

\textbf{Notation.}
We set $\N_0 := \{0,1,2,...\}$. The usual Euclidean norm and scalar product on $\R^k$, $k \geq 1$, are $|\cdot|$ and $"\cdot"$. For an $\R$-valued function $f$, we write $f^+ := \max(f,0)$ and $f^- := -\min(f,0)$. If $(A,D(A))$, is an operator, $R(A)$ denotes its range. If $X$ is a topological space, we denote its Borel $\sigma$-algebra by $\Bscr(X)$. 

For $\vrho: \R^d \to \R_+$ measurable and $p \in [1,\infty)$, we denote the usual spaces of $p$-fold $\vrho dx$-integrable vector fields from $\R^d$ to $\R^k$ by $L^p(\R^d,\R^k;\vrho)$, and their usual norms by $|\cdot|_{p,\vrho}$. If $k=1$, we simply write $L^p(\vrho)$, and $L^p(\R^d,\R^k)$ ($L^p$, for $k = 1$), if $\vrho \equiv 1$. In the Hilbert space case $p=2$, the scalar product is denoted $\langle \cdot, \cdot \rangle_{2,\vrho}$. For the corresponding local spaces, we always suppress the dependence on $\vrho$, since we will only consider the case $0<\vrho \leq 1$. For the same reason, for $p=\infty$ we write $L^\infty(\R^d,\R^k)$ ($L^\infty$, for $k = 1$) independent of the measure $\vrho dx$, and $|\cdot|_\infty$ for its norm (which is independent of $\vrho$).
$(W^{m,p}(\vrho),|\cdot|_{W^{m,p}(\vrho)})$ are the usual Sobolev spaces of $m$-times weakly differentiable functions $g: \R^d \to \R$ with $p$-fold $\vrho dx$-integrable derivatives of order up to $m$. If $\vrho \equiv 1$, we simply write $W^{m,p}$. For $p=2$, we write $H^m(\vrho)$ and $H^m$ instead. For the latter spaces, we denote by $H^{-m}(\vrho)$ and $H^{-m}$ their topological dual spaces, respectively.

For $m \in \N_0 \cup \{\infty\}$, we denote the spaces of $m$-times differentiable functions $g: U\subseteq \R^k \to \R$ (with bounded or compactly supported derivatives up to order $m$) by $C^m(U)$ ($C^m_b(U), C^m_c(U)$, respectively). For $k=d$ and $U=\R^d$, we simply write $C^m, C^m_b$, $C^m_c$.  More generally, for a topological space $X$ we let $C(U,X)$ denote the space of continuous maps $g: U \to X$. The usual space of distributions on $\R^d$ (in duality with $C^\infty_c)$ is $\mathcal{D}'$.

$\Mscr_b$ is the set of non-negative, finite Borel measures on $\R^d$, and $\Pscr$ its subset of Borel probability measures. The weak topology on $\Mscr_b$ is the initial topology of the maps $\mu \mapsto \int_{\R^d}\varphi \,d\mu$, $\varphi \in C_b$. We also set
$$\Pscr_\infty(\vrho):= \bigg\{u \in L^1(\vrho)\cap L^\infty, u \geq 0, \int_{\R^d} u\, \vrho dx=1\bigg\}.$$

Finally, we stress that by $f^{-1}$ we usually denote the inverse of a map $f$, but for the case $f=\vrho$ with $\vrho$ as fixed at the beginning of the next section, we mean $\vrho^{-1} := \frac 1 \vrho$.

\section{Mild solutions to nonlinear Fokker--Planck equations via weighted semigroup approach}\label{sect2}
Throughout, we fix an integer $d \geq 1$. Let $\Phi: \R^d \to \R$, $\vrho := e^{-\Phi}$ and consider the nonlinear Fokker--Planck equation \eqref{equation-main}
with Cauchy initial condition $u(0) = u_0$,
where $u_0: \R^d \to \R$.
Here, $-\divvrho$ denotes the dual operator to the gradient $\nabla$ in $L^2(\vrho)$, i.e. for $F\in L^2(\R^d,\R^d;\vrho)$, $\divvrho F \in H^{-1}(\vrho)$ is defined via
\begin{equation*}
-\int_{\R^d} (\divvrho F) g \,\vrho dx = \int_{\R^d} F \cdot \nabla g \,\vrho dx, \quad \forall g \in H^1(\vrho).
\end{equation*}
A heuristic integration by parts on the right-hand side suggests $\divvrho F = \divv F + F\cdot \nabla (\log \vrho)$, where $\divv$ denotes the usual divergence. The latter equality holds in $L^2(\vrho)$, if all terms are well-defined in $L^2(\vrho)$. Note that $\vrho dx$ is Fomin differentiable, if $\nabla (\log \vrho) = -\nabla \Phi \in L^1(\R^d,\R^d;\vrho)$. We work under the following assumptions.
\paragraph{Hypothesis 1}
\begin{enumerate}
	\item[(H1)] $\Phi \in C^2$ is non-negative and  convex with $\lim_{|x|\to \infty}\Phi(x)=\infty$, 
	and $\nabla \Phi \in L^1(\vrho)$.
	\item[(H2)] $\beta \in C^2(\R)$, $\beta'(r)>0$ for $r \neq 0$, and $\beta(0) = 0$. 
	\item[(H3)] $b\geq 0$, $b \in C_b(\R)$.
	\item[(H4)] $D \in L^\infty(\R^d,\R^d)$, $\divvrho D \in L^2_{\text{loc}} $, $(\divvrho D)^-\in L^\infty$.
\end{enumerate}
\begin{rem}
	\begin{enumerate}
		\item [(i)] By (H1), we have $\vrho \in C^2$, and, since (H1) also implies the existence of $a \in \R$ and $b >0$ such that $\Phi(x)\geq a+b|x|$, one has $\int_{\R^d} \vrho\,dx <\infty$, i.e. $\vrho dx$ is a finite measure.
		\item[(ii)] Moreover,  since (H1) implies $0< \vrho \leq 1$ and $\inf_{x \in K}\vrho(x) >0$ for all compacts $K \subseteq \R^d$, the embedding $W^{m,p}\subseteq W^{m,p}(\vrho )$ is continuous for all $m \in \N_0$, $p \in [1,\infty)$.
		\item[(iii)] $\nabla \Phi \in L^1(\R^d,\R^d;\vrho)$ is equivalent to $\nabla \vrho \in L^1(\R^d,\R^d)$.
	\end{enumerate}
\end{rem}
We will construct \textit{generalized solutions} to \eqref{equation-main} in the following sense.
\begin{dfn}\label{def:gen-sol}
	\begin{enumerate}
		\item [(i)] 	Let $X$ be a real Banach space, $(\tilde{A},D(\tilde{A}))$ an operator $\tilde{A}: D(\tilde{A})\subseteq X\to X$, and $u_0 \in X$. A continuous function $u: [0,\infty)\to X$ is called \emph{generalized (\textup{or} mild) solution} to the Cauchy problem
		\begin{equation}\label{eq:gen-nonlinear-ev-eq}
			\frac{d}{dt}u + \tilde{A}u = 0, \quad u(0) = u_0,
		\end{equation}
		if there is $h_0>0$ such that for all $0<h<h_0$ there is a step function $u_h$, defined by the implicit finite difference scheme
		\begin{align}\label{eq:fin-diff-scheme}
			&\notag u_h(0) :=u^0_h :=  u_0, \\
			&u_h(t) := u^i_h, \forall t \in ((i-1)h,ih],i \in \N,\\
			&\notag u^i_h \in D(\tilde{A}), u^{i}_h + h\tilde{A}u^i_h = u^{i-1}_h,
		\end{align}
		such that 
		\begin{equation*}
			u(t) = \lim_{h \to 0}u_h(t)\text{ in }X,\text{ locally uniformly in }t.
		\end{equation*}
		\item[(ii)] $u$ is  a \textit{generalized (mild) solution to \eqref{equation-main}} with initial datum $u_0\in L^1(\vrho)$, if $u$ is a solution in the sense of (i) with $X=L^1(\vrho)$ and $(\tilde{A},D(\tilde{A})=(A_0,D(A_0))$, where $(A_0,D(A_0))$ is defined in \eqref{def:op-A0} below.
	\end{enumerate}
	
\end{dfn}
If $(\tilde{A},D(\tilde{A}))$ is $m$-accretive, then for each $u_0 \in \overline{D(\tilde{A})}$ (where $\overline{D(\tilde{A})}$ denotes the closure of $D(\tilde{A})$ in $X$), there exists a unique generalized solution $u=u(u_0)$ to \eqref{eq:gen-nonlinear-ev-eq}, see \cite{CL71, B12-Mon-Op-book}, the approximations $u_h$ are given by
\begin{equation*}
	u_h(t) = (I+h\tilde{A})^{-i}u_0, \quad t \in ((i-1)h,ih],
\end{equation*}
and the exponential formula
\begin{equation*}\label{exp-formula}
	u(t) = \lim_{n \to \infty}\bigg(I+\frac t n \tilde{A}\bigg)^{-n}u_0
\end{equation*}
holds locally uniformly in $t$.
In order to relate mild solutions $u$ to \eqref{equation-main} to solutions to the corresponding McKean--Vlasov SDE \eqref{eq:SDE}, we need to show that $u$ is also a \textit{distributional solution} to \eqref{equation-main} in the following sense.
\begin{dfn}\label{def:distr-sol}
	A measurable curve  $\nu: [0,\infty) \to \Mscr_b$ is a \emph{distributional solution} to \eqref{equation-main} with initial datum $\nu_0 \in \Mscr_b$, if $\nu(t) = v(t) dx$ such that $v,\beta(v\vrho^{-1})\vrho\in L^1_{\loc}([0,\infty)\times \R^d; dxdt)$ and for all $\varphi \in C^\infty_c$ there is a set of full $dt$-measure $T_\varphi \subseteq (0,\infty)$ such that for all $t \in T_\varphi$
	\begin{equation}\label{eqaux0}
		\int_{\R^d} \varphi \,v(t) dx = \int_{\R^d} \varphi \,d\nu_0 + \int_0^t \int_{\R^d} \frac{\beta(v(s)\vrho^{-1})}{v(s)\vrho^{-1}}\Delta\varphi +\bigg(b(v(s)\vrho^{-1})D -\frac{\beta(v(s)\vrho^{-1})}{v(s)\vrho^{-1}} \nabla \Phi \bigg)\cdot \nabla \varphi\,d\nu(s) ds.
	\end{equation}
We call $\nu$ a \textit{probability solution}, if $\nu(t)$ is a probability measure for $dt$-a.e. $t >0$.
\end{dfn}
The local integrability assumptions from the previous definition are fulfilled, if $v(t)=u(t)\vrho$, where $u$ is the mild solution from Theorem \ref{main-thm}. If $v: t \mapsto \nu(t)$ is vaguely continuous, one can choose $T_\varphi = (0,\infty)$ for all $\varphi \in C_c^\infty$.
The following theorem is our main existence result for generalized and distributional solutions to \eqref{equation-main}.
\begin{theorem}\label{main-thm}
	Let Hypothesis 1 be fulfilled. For each $u_0 \in L^1(\vrho)\cap L^\infty$ there exists a generalized solution $u = u(u_0) \in C([0,\infty),L^1(\vrho))$ to \eqref{equation-main} with initial datum $u_0$ such that
	\begin{equation}\label{eq:exp-bound-infty-norm}
	|u(t)|_\infty \leq \exp\big(|(\divvrho D)^-+|D||_\infty^{\frac 1 2}t\big)|u_0|_\infty,\quad \forall t \geq 0,
	\end{equation}
	and $t\mapsto u(t)\vrho dx$ is also a weakly continuous distributional solution to \eqref{equation-main}. Moreover, $t \mapsto S(t)u_0 := u(u_0)(t)$ is a semigroup of $L^1(\vrho)$-contractions on $L^1(\vrho)\cap L^\infty$, i.e. for all $u_0,v_0 \in L^1(\vrho)\cap L^\infty$ and $t,s \geq 0$
		\begin{align}
		&S(t+s)u_0 = S(t)S(s)u_0, \,\, S(0) = \Id,\label{sg-prop}\\
		&|S(t)u_0-S(t)v_0|_{1,\vrho} \leq |u_0-v_0|_{1,\vrho}\notag.
	\end{align}
		Furthermore, if $u_0 \in \Pscr_\infty(\vrho)$, then 
		\begin{equation}\label{prop-measure-preserving}
			u(t) \in \Pscr_\infty(\vrho),\quad \forall t \geq 0.
		\end{equation}
\end{theorem}
Moreover, if in addition to (H1)-(H4) also (H2') holds, where
\begin{enumerate}
\item[(H2')] $\beta$ is Lipschitz continuous,
\end{enumerate}
we get the following stronger result.
\begin{prop}\label{prop:prop1}
	If (H1)-(H4) and (H2') are satisfied, there is a generalized and distributional solution $u=u(u_0)$ to \eqref{equation-main} for every $u_0 \in L^1(\vrho)$, and the semigroup- and $L^1(\vrho)$-contraction property of $(S(t))_{t \geq 0}$, $S(t)(u_0) = u(u_0)(t)$, from Theorem \ref{main-thm} extend to $L^1(\vrho)$. Moreover, if $u_0 \in L^1(\vrho)$ such that $u_0 \vrho dx \in \Pscr$, then $S(t)u_0\,\vrho dx \in \Pscr$ for all $t \geq 0$.
\end{prop}
The proofs of Theorem \ref{main-thm} and Proposition \ref{prop:prop1} are given in the remainder of the present and the following section. First, we introduce the operators $(L,D(L))$ and $(L_0, D(L_0))$ as follows.
$$L: D(L)\to L^2(\vrho), \quad Lf:= \Delta f - \nabla \Phi \cdot \nabla f,\quad D(L) := \{f \in H^2(\vrho)\,|\,\nabla \Phi \cdot \nabla f \in L^2(\vrho)\}.$$ By Lemma \ref{lem-all-symmetry-eq-L} below, $(L,D(L))$ is symmetric on $L^2(\vrho)$.
Since $\Phi \in C^2$, $\Delta - \nabla \Phi \cdot \nabla$ can also be considered a continuous operator  $L_0$,
$$L_0: D(L_0) := L^1_{\textup{loc}} \to \mathcal{D}', L_0 f = \Delta f - \nabla \Phi \cdot \nabla f.$$ Then, $L = L_0$ on $D(L)\subseteq D(L_0)$ in the sense that $Lf = L_0f$ as elements in $\Dscr'$ for $f \in D(L)$.
Consider the nonlinear operator $(A_0,D(A_0))$ in $L^1(\vrho)$
\begin{align}\label{def:op-A0}
A_0(f) &:=  -\Delta \beta(f)+\nabla \Phi \cdot \nabla \beta(f)+\divvrho(Db(f)f) =-L_0\beta(f)+\divvrho(Db(f)f),\\
D(A_0) &:= \big\{f \in L^1(\vrho)\cap L^\infty\,|\, -L_0\beta(f)+\divvrho(Db(f)f) \in L^1(\vrho)\cap L^\infty\big\}.\notag
\end{align}
Note that  $f \in D(A_0)$ implies $\beta(f) \in L^1(\vrho)\cap L^\infty$ (since $\beta$ is locally Lipschitz continuous), and hence $L_0\beta(f) \in \Dscr'$.
The following lemma is the key part to prove Theorem \ref{main-thm}.
\begin{lem}\label{main-lemma}
 Let $\lambda_0 := [(|(\divvrho)^-+|D||_\infty+|(\divvrho)^-+|D||_\infty^{\frac 1 2})|b|_\infty]^{-1}$.  We have 
	\begin{enumerate}
		\item [(i)] $R(I+\lambda A_0) = L^1(\vrho)\cap L^\infty$, $\forall 0<\lambda <\lambda_0$.
		\item[(ii)] For each $0<\lambda < \lambda_0$, there is $J_\lambda: L^1(\vrho)\cap L^\infty \to D(A_0)\subseteq L^1(\vrho)\cap L^\infty$, $J_\lambda f \in (I+\lambda A_0)^{-1}f$ such that
		\begin{equation*}\label{contraction-statement}
			|J_\lambda f -J_\lambda g |_{1,\vrho}\leq |f-g|_{1,\vrho}, \quad\forall f,g \in L^1(\vrho)\cap L^\infty.
		\end{equation*}
	Moreover, for all $f\in L^1(\vrho)\cap L^\infty$ and $0<\lambda_1,\lambda_2,\lambda<\lambda_0$
	\begin{equation}\label{resolvent-eq-statement}
		J_{\lambda_2}f = J_{\lambda_1}\bigg(\frac{\lambda_1}{\lambda_2}f + \big(1-\frac{\lambda_1}{\lambda_2}\big)J_{\lambda_2}f\bigg),
	\end{equation}
\begin{equation}\label{infty-bound-statement}
	|J_\lambda f|_\infty \leq(1+|(\divvrho D)^-+|D||_\infty^{\frac 1 2}) |f|_\infty,
\end{equation}
and 
\begin{equation}\label{prop-measure-lemma}
	J_\lambda (\Pscr_\infty(\vrho)) \subseteq \Pscr_\infty(\vrho).
\end{equation}
Finally,
\begin{equation}\label{eq:density-of-L1-Linfty-in-DA}
|J_\lambda g - g|_{1,\vrho} \leq C\lambda (|g|_{H^2(\vrho)}+|g|_{W^{2,1}(\vrho)}),\quad \forall g \in C^\infty_c,
\end{equation}
for a constant $C>0$ independent of $\lambda$.

	\end{enumerate}
\end{lem}
The proof of Lemma \ref{main-lemma} is given in the next section. In the rest of the present section, we prove Theorem \ref{main-thm} as follows.
Define the operator $A: D(A)\subseteq L^1(\vrho)\cap L^\infty\to L^1(\vrho)\cap L^\infty$ via
\begin{align}\label{def:op-A}
	A(f) &:= A_0(f), \\
\notag 	D(A) &:= \big\{J_\lambda f, f \in L^1(\vrho)\cap L^\infty\big\},
\end{align}
for any $0<\lambda < \lambda_0$. Indeed, by \eqref{resolvent-eq-statement}, $D(A)$ does not depend on $\lambda$. Any generalized solution of
\begin{equation}\label{eq:equation-for-A}
	\frac{d}{dt}u +Au = 0
\end{equation}
is also a generalized solution to the same equation with $A_0$ replacing $A$, and hence a generalized solution to \eqref{equation-main}.
By definition, $A$ is a restriction of $A_0$ such that $R(I+\lambda A) = R(I+\lambda A_0)  = L^1(\vrho)\cap L^\infty$ for all $0<\lambda<\lambda_0$, but in contrast to $A_0$, $A$ is defined such that $I+\lambda A$ is bijective from $D(A)$ to $L^1(\vrho)\cap L^\infty$.
Indeed, we have the following lemma.
\begin{lem}\label{final-lem}
For each $0<\lambda<\lambda_0$, $I+\lambda A$ is bijective from $D(A)$ to $L^1(\vrho)\cap L^\infty$ with inverse $J_\lambda$, and 
\begin{equation*}\label{non-expansive-eq}
|(I+\lambda A)^{-1} f- (I+\lambda A)^{-1} g|_{1,\vrho} \leq |f-g|_{1,\vrho}, \quad \forall f,g \in L^1(\vrho)\cap L^\infty.
\end{equation*}
Moreover, \eqref{resolvent-eq-statement}-\eqref{prop-measure-lemma} hold with $(I+\lambda A)^{-1}$ instead of $J_\lambda$, and $L^1(\vrho)\cap L^\infty \subseteq \overline{D(A)}$, where $\overline{C}$ denotes the $L^1(\vrho)$-closure of a set $C \subseteq L^1(\vrho)$.
\end{lem}
\begin{proof}[Proof of Lemma \ref{final-lem}]
	All assertions but the final one are immediate consequences of Lemma \ref{main-lemma}. For the final assertion, letting $\lambda \longrightarrow 0$ in \eqref{eq:density-of-L1-Linfty-in-DA}, the definition of $D(A)$ yields $C^\infty_c\subseteq\overline{ D(A)}$, and hence the claim follows from the density of $C^\infty_c$ in $L^1(\vrho)\cap L^\infty$ with respect to $|\cdot|_{1,\vrho}$.
\end{proof}
Then the proof of Theorem \ref{main-thm}  is obtained as follows.
\\
\\
\textit{Proof of Theorem \eqref{main-thm}.}
The existence of a semigroup $(S(t))_{t\geq 0}$ of $L^1(\vrho)$-contractions on $L^1(\vrho)\cap L^\infty$ such that $t\mapsto S(t)u_0$ is a generalized solution to \eqref{eq:equation-for-A} (and hence to \eqref{equation-main}) with initial datum $u_0$ for each $u_0 \in L^1(\vrho)\cap L^\infty$ and 
\begin{equation}\label{eq:conv-exp-form-proof-mainthm}
S(t)u_0 = \lim_{n \to \infty}\bigg(I+\frac t n A\bigg)^{-n}u_0 = \lim_{n \to \infty}(J_{\frac t n})^{n}u_0 \text{ in }L^1(\vrho)\text{ locally uniformly in }t\geq 0
\end{equation} follows from Lemma \ref{final-lem} and the Crandall-Liggett theory for nonlinear evolution equations in Banach spaces, see \cite[Thm.4.3]{B12-Mon-Op-book}. Indeed, even though Lemma \ref{final-lem} does not imply $m$-accretivity of $(A,D(A))$ in $L^1(\vrho)$, the proof of \cite[Thm.4.3.]{B12-Mon-Op-book} shows that it still implies the claims above. By \eqref{infty-bound-statement}, standard arguments entail \eqref{eq:exp-bound-infty-norm}. Moreover, if $u_0 \in \Pscr_\infty(\vrho)$, \eqref{prop-measure-preserving} follows from \eqref{prop-measure-lemma} and \eqref{eq:conv-exp-form-proof-mainthm}. Since $t \mapsto u(t)\vrho dx$ is weakly continuous due to the continuity of $t\mapsto u(t)$ in $L^1(\vrho)$, it remains to prove that $t\mapsto u(t)\vrho dx$ is also a distributional solution. This we prove after the proof of Lemma \ref{main-lemma} at the end of Section \ref{sect:proof-main-lemma}.

\begin{rem}\label{rem:rem}
	Let us comment on the question of uniqueness of $u$ in the class of generalized and distributional solutions to \eqref{equation-main}. 
	\begin{enumerate}
		\item [(i)] As said before, Lemma \ref{final-lem} does not yield $m$-accretivity of $(A,D(A))$ in either $L^1(\vrho)$ or $L^1(\vrho)\cap L^\infty$. Hence the Crandall-Liggett theorem does not yield uniqueness of $u$ in the class of generalized solutions to \eqref{equation-main}. $u$ is, however, the unique generalized solution such that the approximations $u_h$ in Definition \ref{def:gen-sol} can be chosen so that $\sup_{0< h<h_0}|u_h(t)|_\infty$ is locally bounded in $t$ for some $h_0 >0$. Indeed, this follows by Lemma \ref{final-lem} and an inspection of the uniqueness part of the proof of \cite[Thm.4.1]{B12-Mon-Op-book}.
		However, even with $m$-accretivity of $(A,D(A))$, uniqueness of $u$ would only hold in the class of generalized solutions \emph{with respect to $(A,D(A))$}, which is a restriction of $(A_0,D(A_0))$, but not in the class of \emph{all} generalized solutions.
		\item [(ii)] For $\vrho \equiv1$ and under additional assumptions on the coefficients, it is proven in  \cite{BR22} that $A_0$ itself is $m$-accretive. In this case, the generalized solutions are unique. We believe that such a result can be extended to more general weights $\vrho$, but we do not pursue this question in the present paper.
	\end{enumerate}
\end{rem}
For the proof of Lemma \ref{main-lemma}, we need the following result.
\begin{lem}\label{lem-all-symmetry-eq-L}
	\begin{enumerate}
		\item [(i)] 	Let $v\in D(L)$ and $w \in H^1(\vrho)$. Then $	\int_{\R^d} Lv \,w \,\vrho dx = -\int_{\R^d} \nabla v \cdot \nabla w \,\vrho dx.$
	\item[(ii)] If either 
	\begin{itemize}
		\item [(a)] $v,w\in D(L)$ or
		\item[(b)] $v\in L^1_{\textup{loc}}$ such that $L_0 v \in L^1_{\textup{loc}}$ and $w \in C^2_c$, then $	\int_{\R^d} w\,L_0v \, \vrho dx = \int_{\R^d} v Lw \,\vrho dx.$
	\end{itemize}
	\end{enumerate}
In particular, $(L,D(L))$ is symmetric on $L^2(\vrho)$.
\end{lem}

\begin{proof}
	\begin{enumerate}
	\item[(i)] For $w \in C^1_c$, the claim follows from integration by parts to $(\Delta v)w\vrho$, noting that $w\vrho \in C^1_c$ and $-\nabla v \cdot \nabla (w \vrho) = -\nabla v \cdot (\nabla w -w\nabla \Phi)\vrho$. Since $C^1_c$ is dense in $H^1(\vrho)$, the claim follows.
	\item [(ii)] For (a), repeat the proof of (i), then reverse the roles of $v$ and $w$ to get
	\begin{equation*}
	\int_{\R^d} w \,L_0v \,\vrho dx = 	\int_{\R^d} wLv \,\vrho dx = -\int_{\R^d}\nabla v\cdot \nabla w \,\vrho dx = \int_{\R^d} vLw \,\vrho dx.
	\end{equation*}
For (b), as a distribution $L_0v$ acts on $w\vrho \in C_c^{2}$  via $\int w\, L_0v \,\vrho dx$. Also,  $\Delta v$ and $-\nabla \Phi \cdot \nabla v$ are distributions, and the action of their sum applied to $\omega \vrho $ is
\begin{align*}
	\int_{\R^d} v([\Delta(w \vrho) + \divv(w\vrho \nabla \Phi) ] dx = \int_{\R^d} v [\Delta w-\nabla \Phi \cdot \nabla w]\vrho dx = \int_{\R^d} v Lw \,\vrho dx.
\end{align*}
\end{enumerate}
\end{proof}
\section{Proof of Lemma \ref{main-lemma}}\label{sect:proof-main-lemma}

Let $\lambda>0$.
The definition of $(A_0,D(A_0))$ entails $R(I+\lambda A_0) \subseteq L^1(\vrho)\cap L^\infty$. Let us prove the reverse inclusion.
In order to solve 
\begin{equation*}\label{main-eq-lemma}
u+\lambda A_0u = f,
\end{equation*}
we approximate $\beta, D$ and $b$ as follows. For $\varepsilon >0$, consider
\begin{equation*}
	\tilde{\beta}_{\varepsilon}(r):= \beta_{\varepsilon}(r)+\varepsilon r := \frac 1 \varepsilon\big(r- (I+\varepsilon\beta)^{-1}r\big)+\varepsilon r = \beta\big((I+\varepsilon\beta)^{-1}r\big)+\varepsilon r, \,r\in \R.
\end{equation*}
It is readily seen that $\tilde{\beta}_{\varepsilon}$ is strictly increasing and, for all $r,t \in \R$,
\begin{align}
	&|\tilde{\beta}_{\varepsilon}(r)-\tilde{\beta}_{\varepsilon}(t)| \leq \big(\varepsilon+ 2 \varepsilon^{-1}\big)|r-t|, \label{betatildeepsilon-lipschitz}\\
	&|\tilde{\beta}_{\varepsilon}(r)-\tilde{\beta}_{\varepsilon}(t)| \geq \varepsilon|r-t|, \label{betatildeepsilon-lipschitz-2}\\
	& \tilde{\beta}_{\varepsilon}(0) = 0\label{0 in 0},
\end{align}
i.e. $\tilde{\beta}_{\varepsilon}:\R \to \R$ is a bijective $(\varepsilon+\frac 2 \varepsilon)$-Lipschitz map with $\frac 1 \varepsilon$-Lipschitz inverse.  Set
\begin{equation*}
	b_\varepsilon(r):= \frac{(b*\rho_\varepsilon)(r)}{1+\varepsilon|r|}, \quad r \in \R,
\end{equation*}
and $
D_\varepsilon := \eta_\varepsilon D,	 
$
where $\rho_\varepsilon(r):= \varepsilon^{-1}\rho(\frac r \varepsilon), \rho \in C^\infty_c, \rho \geq 0$ is a standard mollifier, and $(\eta_\varepsilon)_{\varepsilon>0} \subseteq C^1_c$ is a family of functions with $0\leq \eta_\varepsilon \leq 1$, $|\nabla \eta_\varepsilon|\leq 1$ and $\eta_\varepsilon \equiv 1$ on $B_{\varepsilon^{-1}}(0)$ (the Euclidean ball with radius $\varepsilon^{-1}$ centered at $0$). Consequently, $b^*_\varepsilon(r):= b_\varepsilon(r)r$ is bounded and Lipschitz, and $b_\varepsilon(r) \longrightarrow b(r)$ pointwise as $\varepsilon\to 0$. Moreover, by (H4),
$$\divvrho D_\varepsilon  = \eta_\varepsilon \divvrho D + \nabla \eta_\varepsilon \cdot D \in L^2(\vrho) + L^\infty.$$
Consider the approximate equation
\begin{equation}\label{approx-eq-eps}
u- \lambda L \tilde{\beta}_{\varepsilon}(u) + \lambda \varepsilon\tilde{\beta}_{\varepsilon}(u)+\lambda \divvrho(D_\varepsilon b^*_\varepsilon(u))= f.
\end{equation}
To solve it, first let $f \in L^2(\vrho)$, consider in $L^2(\vrho)$
\begin{equation}\label{eq-1}
	(\varepsilon I- L)^{-1} u +\lambda \tilde{\beta}_{\varepsilon}(u) + \lambda (\varepsilon I - L)^{-1}\divvrho(D_\varepsilon b^*_\varepsilon(u))= (\varepsilon I - L)^{-1}f,
\end{equation}
and define on $L^2(\vrho)$ the maps $F_\varepsilon$, $G^1_\varepsilon$, $G^2_\varepsilon$
\begin{align*}
	F_{\varepsilon}: u \mapsto 	(\varepsilon I- L)^{-1}u, \quad G^1_\varepsilon: u \mapsto \lambda \tilde{\beta}_\varepsilon(u),\quad 	G^2_\varepsilon: u \mapsto \lambda (\varepsilon I - L)^{-1}\divvrho(D_\varepsilon b^*_\varepsilon(u)),
\end{align*}
so that \eqref{eq-1} becomes
\begin{equation*}
	(F_\varepsilon+G^1_\varepsilon+G^2_\varepsilon)u = (\varepsilon I-L)^{-1}f.
\end{equation*}
By \eqref{betatildeepsilon-lipschitz}, $G^1_\varepsilon$ is $L^2(\vrho)$-continuous, and \eqref{betatildeepsilon-lipschitz-2} and the monotonicity of $\beta_\varepsilon$ entail
\begin{equation*}
	\langle G^1_\varepsilon (u-v),u-v\rangle_{2,\vrho}\geq \lambda\varepsilon|u-v|^2_{2,\vrho}, \quad \forall u,v \in L^2(\vrho).
\end{equation*}
In \cite{DaPL04} it is shown that the resolvent set of $(L,D(L))$
contains $(0,\infty)$ and, moreover,
\begin{equation*}
|(\varepsilon I-L)^{-1}u|_{2,\vrho} \leq \varepsilon^{-1}|u|_{2,\vrho}, \quad \forall u \in L^2(\vrho).
\end{equation*}
In particular, $F_\varepsilon$ is continuous on $L^2(\vrho)$. Furthermore, Lemma \ref{lem-all-symmetry-eq-L} (i) yields
\begin{equation*}
\langle F_\varepsilon (u), u \rangle_{2,\vrho} = \varepsilon |(\varepsilon I -L )^{-1} u|^2_{2,\vrho}+|\nabla (\varepsilon I - L)^{-1} u|^2_{2,\vrho},\quad \forall u \in L^2(\vrho).
\end{equation*}
Since $r \mapsto b^*_\varepsilon(r)$ is Lipschitz and $D_\varepsilon\in L^\infty$, $u \mapsto D_\varepsilon b^*_\varepsilon(u)$ is $L^2(\vrho)$-continuous and hence $G^2_\varepsilon$ is $L^2(\vrho)$-continuous as well. Moreover, denoting by $C_\varepsilon>0$ a constant depending on $|D_\varepsilon|_\infty$ and the Lipschitz constant of $b^*_\varepsilon$, we have 
\begin{align*}
	\langle G^2_\varepsilon(u)-G^2_\varepsilon(v),u-v\rangle_{2,\vrho} &= -\lambda \langle D_\varepsilon ( b^*_\varepsilon(u)-b^*_\varepsilon(v)) , \nabla (\varepsilon I - L)^{-1}(u-v)\rangle_{2,\vrho}
	\\& \geq -C_\varepsilon \lambda |u-v|_{2,\vrho}|\nabla (\varepsilon I - L)^{-1} (u-v)|_{2,\vrho},\quad \forall u,v \in L^2(\vrho).
\end{align*}
Altogether, these estimates lead to 
\begin{align*}
&	\langle F_\varepsilon(u-v)+G^1_\varepsilon(u)-G^1_\varepsilon (v)+G^2_\varepsilon (u)-G^2_\varepsilon (v), u-v\rangle_{2,\vrho}
\\& \geq \lambda\varepsilon|u-v|^2_{2,\vrho} + \varepsilon |(\varepsilon I -L )^{-1} (u-v)|^2_{2,\vrho}+|\nabla (\varepsilon I - L)^{-1} (u-v)|^2_{2,\vrho} \\& \quad \quad\quad-C_\varepsilon \lambda |u-v|^2_{2,\vrho}|\nabla (\varepsilon I - L)^{-1} (u-v)|^2_{2,\vrho}
 \geq \frac{\lambda \varepsilon}{2}|u-v|^2_{2,\vrho},
\end{align*}
provided $0< \lambda < \lambda_\varepsilon$ for $\lambda_\varepsilon>0$ sufficiently small,
i.e. $F_\varepsilon+G^1_\varepsilon+G^2_\varepsilon$ is coercive on $L^2(\vrho)$. Since the continuity and monotonicity (= accretivity) of $F_\varepsilon+G^1_\varepsilon+G^2_\varepsilon$ on $L^2(\vrho)$ also imply its $m$-accretivity, it follows by \cite[Cor.2.2.]{B12-Mon-Op-book} that $F_\varepsilon+G^1_\varepsilon+G^2_\varepsilon$ is surjective (hence bijective) on $L^2(\vrho)$. Therefore, for each $f \in L^2(\vrho)$, \eqref{eq-1} has a unique solution $u_\varepsilon (f)= u_\varepsilon(\lambda, f)$ in $L^2(\vrho)$. Since for this solution the right-hand side and the first summand on the left-hand side of \eqref{eq-1} are in $D(L) \subseteq H^2(\vrho)$, and $G^2_\varepsilon(u_\varepsilon) \in H^1(\vrho)$, it follows that $\tilde{\beta}_\varepsilon(u_\varepsilon)\in H^1(\vrho)$. Since $\tilde{\beta}_\varepsilon$ has a Lipschitz inverse, we obtain $u_\varepsilon \in H^1(\vrho)$. Thus also $b^*_\varepsilon(u_\varepsilon) \in H^1(\vrho)$. Since 
$$\divvrho(D_\varepsilon b^*_\varepsilon(u_\varepsilon)) = b^*_\varepsilon(u_\varepsilon)\divvrho D_\varepsilon + D_\varepsilon \cdot \nabla b^*_\varepsilon(u_\varepsilon)$$
and $b^*_\varepsilon(u_\varepsilon) \in L^2(\vrho)\cap L^\infty$, $\divvrho D_\varepsilon \in L^2(\vrho) + L^\infty$,  $D_\varepsilon\in L^\infty$, we obtain $\divvrho(D_\varepsilon b^*_\varepsilon(u_\varepsilon)) \in L^2(\vrho)$, and therefore even $G^2_\varepsilon(u_\varepsilon) \in D(L)$. Hence, all terms of \eqref{eq-1} are in $D(L)$, and applying $(\varepsilon I - L)$ to both sides shows that $u_\varepsilon$ solves \eqref{approx-eq-eps} in $L^2(\vrho)$.

Next, we prove $u_\varepsilon \in L^1(\vrho)$ if $f\in L^1(\vrho)\cap L^2(\vrho)$ and that for $0<\lambda < \lambda_\varepsilon$
\begin{equation*}
	|u_\varepsilon(\lambda, f_1)-u_\varepsilon(\lambda,f_2)|_{1,\vrho}\leq |f-g|_{1,\vrho},\quad \forall f_1,f_2 \in L^1(\vrho).
\end{equation*}
To this end, set $u:= u_1 - u_2:= u_\varepsilon(\lambda,f_1)-u_\varepsilon(\lambda,f_2)$ and $f:= f_1-f_2$. Then, by \eqref{approx-eq-eps},
\begin{align}\label{eq-2}
	u = \lambda L[\tilde{\beta}_{\varepsilon}(u_1)-\tilde{\beta}_{\varepsilon}(u_2)]-\lambda \varepsilon[\tilde{\beta}_{\varepsilon}(u_1)-\tilde{\beta}_{\varepsilon}(u_2)]-\lambda \divvrho(D_\varepsilon[b^*_\varepsilon(u_1)-b^*_\varepsilon(u_2)])+f.
\end{align}
Consider, for $\delta>0$, the Lipschitzian function $\chi_\delta:\R \to \R$,
\begin{align*}
	\chi_\delta(r):=
	\begin{cases}
	1,& r \geq \delta,\\
	\frac r \delta,& |r| < \delta,\\
	-1,& r < -\delta,
	\end{cases}
\end{align*}
i.e. $\chi_\delta\to \text{sign}$ pointwise for $\delta \to 0$. Multiplying both sides of \eqref{eq-2} by $\Lambda_\delta := \chi_\delta(\tilde{\beta}_{\varepsilon}(u_1)-\tilde{\beta}_{\varepsilon}(u_2))\in H^1(\vrho)$, integrating with respect to $\vrho dx$ and using $\Lambda_{\delta} (\tilde{\beta}_{\varepsilon}(u_1)-\tilde{\beta}_{\varepsilon}(u_2))\geq 0$ implies
\begin{align*}\label{eq-3}
	\int_{\R^d} u \Lambda_\delta \vrho dx \leq \lambda\int_{\R^d} L[\tilde{\beta}_{\varepsilon}(u_1)-\tilde{\beta}_{\varepsilon}(u_2)]\Lambda_\delta \,\vrho dx+ \int_{\R^d}\Lambda_\delta \divvrho(D_\varepsilon [b^*_\varepsilon(u_1)-b^*_\varepsilon(u_2)])\,\vrho dx+\int_{\R^d} f \Lambda_\delta \,\vrho dx .
\end{align*}
By Lemma \ref{lem-all-symmetry-eq-L} (i) and since $\chi_\delta'(r) = \mathds{1}_{\{|r|< \delta\}}\frac 1 \delta$, the first integral on the right-hand side of the previous inequality equals $-\frac 1 \delta \int_{\{|\tilde{\beta}_{\varepsilon}(u_1)-\tilde{\beta}_{\varepsilon}(u_2)|< \delta\}}|\nabla [\tilde{\beta}_{\varepsilon}(u_1)-\tilde{\beta}_{\varepsilon}(u_2)]|^2\vrho dx \leq 0$. Moreover, for a constant $c_0 >0$,
$$-\int_{\R^d}\Lambda_\delta \divvrho(D_\varepsilon [b^*_\varepsilon(u_1)-b^*_\varepsilon(u_2)])\,\vrho dx \leq \frac{c_0} \delta\int_{\{|\tilde{\beta}_{\varepsilon}(u_1)-\tilde{\beta}_{\varepsilon}(u_2)|< \delta\}}|D_\varepsilon| |u| |\nabla (\tilde{\beta}_\varepsilon(u_1)-\tilde{\beta}_\varepsilon(u_2))|\,\vrho dx.$$
Since $\tilde{\beta}_\varepsilon$ has a Lipschitz inverse, we further bound the right-hand side of the previous estimate by
$$c_0|D_\varepsilon|_{2,\vrho}\bigg(\int_{\{|\tilde{\beta}_{\varepsilon}(u_1)-\tilde{\beta}_{\varepsilon}(u_2)|< \delta\}} |\nabla (\tilde{\beta}_\varepsilon(u_1)-\tilde{\beta}_\varepsilon(u_2))|^2\,\vrho dx\bigg)^{\frac 1 2}\xrightarrow{\delta \to 0}0,$$
since $\nabla (\tilde{\beta}_\varepsilon(u_1)-\tilde{\beta}_\varepsilon(u_2)) = 0$ $\vrho dx $-a.s. on $\{|\tilde{\beta}_{\varepsilon}(u_1)-\tilde{\beta}_{\varepsilon}(u_2)| = 0\}$.
Thus, altogether we obtain by Fatou's lemma and since $|\Lambda_\delta|\leq 1$:
\begin{equation}\label{eq:L1-approx-eps-eq}
\int_{\R^d}|u|\,\vrho dx  \leq |f|_{1,\vrho},
\end{equation}
Since $u_\varepsilon(\lambda, 0) \equiv 0$ for all $0< \lambda < \lambda_\varepsilon$, in particular we have shown $|u_\varepsilon|_{1,\vrho}\leq |f|_{1,\vrho}$ for any $f \in L^1(\vrho)\cap L^2(\vrho)$. By \eqref{betatildeepsilon-lipschitz}, then also $\tilde{\beta}_{\varepsilon}(u_\varepsilon) \in L^1(\vrho)$. Consequently, \eqref{approx-eq-eps} implies $-L\tilde{\beta}_{\varepsilon}(u_\varepsilon)+\divvrho(D_\varepsilon b^*_\varepsilon(u_\varepsilon)) \in L^1(\vrho)$, and in this sense \eqref{approx-eq-eps} holds in $L^1(\vrho)$.


Now let $f \in L^1(\vrho)$ and consider the operator $(A_\varepsilon, D(A_\varepsilon))$
\begin{equation*}\label{def:A_vareps}
A_\varepsilon u := - \Delta \tilde{\beta}_{\varepsilon}(u) +\nabla \Phi \cdot \nabla \tilde{\beta}_{\varepsilon}(u) + \varepsilon\tilde{\beta}_{\varepsilon}(u) + \divvrho(D_\varepsilon b^*_\varepsilon(u)),
\end{equation*}
\begin{equation*}
	D(A_\varepsilon) := \{u \in L^1(\vrho): \, A_\varepsilon u \in L^1(\vrho)\},
\end{equation*}
so that we may rewrite \eqref{approx-eq-eps} as
\begin{equation}\label{approx-eq-with-A_epsilon}
u + \lambda A_\varepsilon u  = f.
\end{equation}
Let $(f_n)_{n \in \N} \subseteq L^1(\vrho)\cap L^2(\vrho)$ such that $f_n \longrightarrow f$ in $L^1(\vrho)$ as $n \to \infty$. By \eqref{eq:L1-approx-eps-eq}, the corresponding solutions $u_n:= u_\varepsilon(\lambda, f_n)$ to \eqref{approx-eq-eps} converge to some $u \in L^1(\vrho)$ as $n \to \infty$, and hence $A_\varepsilon(u_n)$ converges in $L^1(\vrho)$ to some $g$. Since $\tilde{\beta}_\varepsilon(u_n)  \in D(L)$, we have $\tilde{\beta}_\varepsilon(u_\varepsilon), \Delta  \tilde{\beta}_\varepsilon(u_\varepsilon), \nabla \Phi \cdot \nabla \tilde{\beta}_\varepsilon(u_\varepsilon) \in L^1_{\loc}$, and the Lipschitz continuity of $b^*_\varepsilon$ and $\tilde{\beta}_\varepsilon$ yields for all $\varphi \in C^2_c$
	\begin{align*}
		\int_{\R^d} (A_\varepsilon u_n) \, \varphi \,\vrho dx
		&=
		\int_{\R^d} \tilde{\beta}_{\varepsilon}(u_n)\big(-L\varphi+\varepsilon\varphi\big) - D_\varepsilon b^*_\varepsilon(u_n)\cdot \nabla \varphi\,\vrho dx \\&
		\converge-n 	\int_{\R^d} \tilde{\beta}_{\varepsilon}(u)\big(-L\varphi+\varepsilon\varphi\big) - D_\varepsilon b^*_\varepsilon(u)\cdot \nabla \varphi \,\vrho dx\\& \quad \quad \quad \quad =
		{}_{\mathcal{D}'}\langle -\Delta \tilde{\beta}_{\varepsilon}(u)+\nabla \Phi \cdot \nabla \tilde{\beta}_\varepsilon(u)+\varepsilon \tilde{\beta}_\varepsilon(u) + \divvrho(D_\varepsilon b^*_\varepsilon(u)) , \varphi \vrho \rangle_\mathcal{D},
	\end{align*}
	where ${}_{\mathcal{D}'}\langle F ,f \rangle_{\mathcal{D}}$ denotes the dual pairing of a second-order distribution $F$ and $f\in C^2_c$. 
Since $\vrho \in C^2(\R^d)$ and $\vrho>0$, one has $\{\varphi \vrho; \varphi \in C^2_c\} = C^2_c$. Consequently, $u \in D(A_\varepsilon)$ and $g = A_\varepsilon u$, i.e. $u$ solves \eqref{approx-eq-with-A_epsilon}. Denoting this (in general non-unique) solution to \eqref{approx-eq-with-A_epsilon} by $u_\varepsilon(\lambda,f)$, \eqref{eq:L1-approx-eps-eq} implies for all $0< \lambda < \lambda_\varepsilon$
\begin{equation}\label{eq:L1-contract-all-f}
	|u_\varepsilon(\lambda,f)-u_\varepsilon(\lambda,g)|_{1,\vrho}\leq |f-g|_{1,\vrho}, \quad \forall f,g \in L^1(\vrho).
\end{equation}
By \cite[Prop.3.3]{B12-Mon-Op-book}, this implies that for each $\varepsilon>0$ equation \eqref{approx-eq-with-A_epsilon} has a solution $u_\varepsilon(\lambda, f) \in D(A_\varepsilon)$ for all $f \in L^1(\vrho)$ for \textit{all} $\lambda >0$, \eqref{eq:L1-contract-all-f} holds for all $\lambda >0$, and 
\begin{equation}\label{eq:resolvent-eq-epsilon}
	u_\varepsilon(\lambda_2, f) = u_\varepsilon\bigg(\lambda_1, \frac{\lambda_1}{\lambda_2}f + (1-\frac{\lambda_1}{\lambda_2})u_\varepsilon(\lambda_2,f)\bigg),\quad \forall 0<\lambda_1,\lambda_2 < \infty, f \in L^1(\vrho).
\end{equation}
Moreover, similarly to the proof of equality (2.29) in  \cite{BRS23}, one proves $u_\varepsilon(\lambda, f) \in L^1(\vrho)\cap L^2(\vrho)$ for all $f \in L^1(\vrho)\cap L^2(\vrho)$ and all $\lambda >0$. Since we know by the above construction that $u_\varepsilon(\lambda,f) \in H^1(\vrho)$ if $\lambda < \lambda_\varepsilon$ and $f \in L^1(\vrho)\cap L^2(\vrho)$, \eqref{eq:resolvent-eq-epsilon} implies $u_\varepsilon(\lambda, f)\in H^1(\vrho)$ for such $f$ for \textit{all} $\lambda >0$.


From now on, let $f \in L^1(\vrho)\cap L^\infty$ and let $u_\varepsilon = u_\varepsilon(\lambda,f) \in H^1(\vrho)\cap L^1(\vrho)$ be the corresponding solution to \eqref{approx-eq-with-A_epsilon}.
For $0 < \lambda < \lambda_0$, where
$$\lambda_0 = \big[\big(|(\divvrho D)^-+|D| |_\infty+ |(\divvrho D)^-+|D| |_\infty^{\frac 1 2}\big)|b|_\infty\big]^{-1},$$
 we are now going to prove
\begin{equation}\label{Linfty-bound-u_eps}
	\sup_{\varepsilon>0}|u_\varepsilon|_\infty \leq  |f|_\infty (1  +|(\divvrho D)^- + |D| |^{\frac 1 2}_\infty)
\end{equation}
as follows. Setting $M_\varepsilon := |(\divvrho D_\varepsilon)^-|_\infty^{\frac 1 2} |f|_\infty$, one has, since $(\divvrho D_\varepsilon)^- \leq |(\divvrho D)^- + |D| |_\infty$, 
\begin{align*}
	u_\varepsilon &- |f|_\infty  -M_\varepsilon -\lambda L\big(\tilde{\beta}_\varepsilon(u_\varepsilon)-\tilde{\beta}_\varepsilon(|f|_\infty+M_\varepsilon)\big) + \lambda \varepsilon \big(\tilde{\beta}_\varepsilon(u_\varepsilon)-\tilde{\beta}_\varepsilon(|f|_\infty+M_\varepsilon)\big) \\&+ \lambda \divvrho \big( D_\varepsilon\big(b_\varepsilon^*(u_\varepsilon)-b_\varepsilon^*(|f|_\infty+M_\varepsilon)\big) \big)\leq f-|f|_\infty-M_\varepsilon + \lambda (\divvrho D_\varepsilon)^-b_\varepsilon^*(|f|_\infty + M_\varepsilon) \leq 0.
\end{align*} Multiplying the above inequality by $\chi_\delta([u_\varepsilon-|f|_\infty + M_\varepsilon]^+)\in H^1(\vrho)$, integrating with respect to $\vrho dx$ and using 
$$\chi_\delta([u_\varepsilon-|f|_\infty-M_\varepsilon]^+)  \lambda \varepsilon \big(\tilde{\beta}_\varepsilon(u_\varepsilon)-\tilde{\beta}_\varepsilon(|f|_\infty+M_\varepsilon)\big) \geq 0$$
as well as
$$-\lambda \int_{\R^d} \chi_\delta([u_\varepsilon-|f|_\infty-M_\varepsilon]^+) L\big(\tilde{\beta}_\varepsilon(u_\varepsilon)-\tilde{\beta}_\varepsilon(|f|_\infty+M_\varepsilon)\big) \,\vrho dx\geq 0,$$
we find
\begin{align*}
	\int_{\R^d} \chi_\delta([u_\varepsilon-|f|_\infty-M_\varepsilon]^+)( u_\varepsilon-|f|_\infty-M_\varepsilon) \,\vrho dx \leq \frac{\lambda}{\delta} \int_{C_\delta}D_\varepsilon\big(b_\varepsilon^*(u_\varepsilon)-b_\varepsilon^*(|f|_\infty+M_\varepsilon)\big)\cdot \nabla u_\varepsilon\,\vrho dx,
\end{align*}
where $C_\delta := \{|f|_\infty  + M_\varepsilon \leq u_\varepsilon \leq |f|_\infty  + M_\varepsilon  + \delta\}$. The right-hand side of the previous inequality is bounded from above by $$\frac{c_\varepsilon\lambda}{\delta}|D_\varepsilon|_{2,\vrho}\bigg(\int_{C_\delta}|u_\varepsilon-|f|_\infty - M_\varepsilon|^2 |\nabla u_\varepsilon|^2\,\vrho dx\bigg)^{\frac 1 2} \leq c_\varepsilon\lambda|D_\varepsilon|_{2,\vrho}\bigg(\int_{C_\delta} |\nabla u_\varepsilon|^2\,\vrho dx\bigg)^{\frac 1 2}\xrightarrow{\delta \to 0}0,$$
where the final convergence follows from $u_\varepsilon \in H^1(\vrho)$ and $\mathds{1}_{C_\delta}\nabla u_\varepsilon\xrightarrow{\delta \to 0}0$ $\vrho dx$-a.s. Consequently,
\begin{align*}
	\int_{\R^d} [u_\varepsilon - |f|_\infty - M_\varepsilon]^+\,\vrho dx \leq \liminf_{\delta \to 0} \int_{\R^d} \chi_\delta([u_\varepsilon-|f|_\infty-M_\varepsilon]^+)( u_\varepsilon-|f|_\infty-M_\varepsilon) \,\vrho dx = 0,
\end{align*}
which implies
\begin{equation*}
	u_\varepsilon \leq |f|_\infty  + M_\varepsilon = |f|_\infty ( 1+ |(\divvrho D_\varepsilon)^-|_\infty^{\frac 1 2}) \leq |f|_\infty (1  +|(\divvrho D)^- + |D| |^{\frac 1 2}_\infty).
\end{equation*}
Similarly, one obtains 
\begin{equation*}
	u_\varepsilon \geq - |f|_\infty (1  +|(\divvrho D)^- + |D| |^{\frac 1 2}_\infty),
\end{equation*}
and so \eqref{Linfty-bound-u_eps} follows.


For the remainder of the proof, we fix $0< \lambda < \lambda_0$, $f \in L^1(\vrho)\cap L^\infty$.
By \eqref{eq:L1-contract-all-f} and \eqref{Linfty-bound-u_eps}, $(u_\varepsilon)_{\varepsilon>0}, u_\varepsilon = u_\varepsilon(\lambda, f)$, is uniformly bounded in each $L^p(\vrho)$, $p \in [1,\infty]$. In particular, there exists an $L^2(\vrho)$-weakly convergent subsequence (for simplicity again denoted $(u_\varepsilon)_{\varepsilon>0}$) and some $u=u(\lambda, f) \in L^2(\vrho)$ such that 
\begin{equation}\label{weak-conv-to-u}
u_\varepsilon  \xrightarrow{\varepsilon \to 0} u \quad L^2(\vrho)-\text{weakly}.
\end{equation}
Moreover,  the uniform $L^\infty$-bound for $\{u_\varepsilon\}_{\varepsilon>0}$ also gives the $\vrho dx$-a.s. estimate
\begin{equation*}
	|\tilde{\beta}_{\varepsilon}(u_\varepsilon)| \leq (\varepsilon+C_{\beta,f})|u_\varepsilon|,
\end{equation*}
where $C_{\beta,f} \in [0,\infty)$ is the Lipschitz constant of $\beta$ on $[-|f|_\infty,|f|_\infty]$. Consequently, $\{\tilde{\beta}_{\varepsilon}(u_\varepsilon)\}_{\varepsilon>0}$ is uniformly $L^p(\vrho)$-bounded for any $p \in [1,\infty]$ as well, so there is $\eta$ such that for all $q \in (1,\infty)$ $\tilde{\beta}_{\varepsilon}(u_\varepsilon) \xrightarrow{\varepsilon\to 0}\eta$ $L^q(\vrho)$-weakly. In fact, also $u_\varepsilon \longrightarrow u$ pointwise $\vrho dx$-a.s., as $\varepsilon \to 0$, as the following argument shows.
Considering \eqref{approx-eq-with-A_epsilon} with $u_\varepsilon \in H^1(\vrho)\cap L^1(\vrho)\cap L^\infty$ instead of $u$, multiplying with $u_\varepsilon$, integrating with respect to $\vrho dx$, using Lemma \ref{lem-all-symmetry-eq-L} (i), the monotonicity of $\tilde{\beta}_\varepsilon$, Young's inequality, and $b_\varepsilon^*(u_\varepsilon) \in H^1(\vrho)$ gives
\begin{equation}\label{a-eq}
	|u_\varepsilon|^2_{L^2(\vrho)} +2\lambda \int_{\R^d} (\beta_{\varepsilon})'(u_\varepsilon)|\nabla u_\varepsilon|^2\,\vrho dx    \leq |f|_{L^2(\vrho)}^2 + 2\lambda\int_{\R^d}D_\varepsilon b_\varepsilon^*(u_\varepsilon)\cdot \nabla u_\varepsilon\,\vrho dx.
\end{equation}
Defining the non-negative function $\psi$,
\begin{align*}
	\psi (r):= \int_0^r b^*_\varepsilon(s)\,ds, \quad r\in \R,
\end{align*}
we rewrite and estimate the final integral on the right-hand side of the previous inequality as
\begin{align}\label{eq:estimate-indep-epsilon}
-2\lambda \int_{\R^d}(\divvrho D_\varepsilon)\psi(u_\varepsilon)\,\vrho dx \leq \lambda |b|_\infty |u_\varepsilon|_\infty |u_\varepsilon|_{1,\vrho}|(\divvrho D)^- + |D||_\infty.
\end{align}
Note that the right-hand side of this estimate is bounded by a finite constant $C$, which is independent of $\varepsilon >0$.
Setting $g_\varepsilon(r):= (I+\varepsilon \beta)^{-1}(r), r \in \R$, one has
\begin{equation*}
	\beta_\varepsilon'(r) \geq \frac{\beta'}{1+\beta'}\big(g_\varepsilon(r)\big)g_\varepsilon'(r)^2,
\end{equation*}
and introducing $a(r):= \int_0^r\frac{\beta'}{1+\beta'}(s)ds, r \in \R$, one has
\begin{equation*}
|\nabla a(g_\varepsilon(u_\varepsilon))|^2 \leq |\nabla u_\varepsilon|^2 \frac{\beta'}{1+\beta'}\big(g_\varepsilon(u_\varepsilon)\big)(g_\varepsilon)'(u_\varepsilon)^2\leq  |\nabla u_\varepsilon|^2\beta_\varepsilon'(u_\varepsilon),
\end{equation*}
and thus we obtain from \eqref{a-eq},\eqref{eq:estimate-indep-epsilon}:
\begin{equation*}
	|u_\varepsilon|^2_{L^2(\vrho)}   + 2\lambda\int |\nabla a(g_\varepsilon(u_\varepsilon))|^2\,\vrho dx 
	 \leq |f|_{L^2(\vrho)}^2 + C.
\end{equation*}
Since also
\begin{equation*}
	|a(g_\varepsilon(r))| \leq |g_\varepsilon(r)| = |(I+\varepsilon\beta)^{-1}(r)| \leq |r|,
\end{equation*}
$\{a(g_\varepsilon(u_\varepsilon))\}_{\varepsilon>0}$ is uniformly bounded in $H^1(\vrho)$, thus compact in $L^2_{\loc}$, so there is $\Lambda \in L^2_{\loc}$ such that for a suitable subsequence, the $L^2_{\loc}$-convergence
\begin{equation*}
	a(g_\varepsilon(u_\varepsilon)) \xrightarrow{\varepsilon \to 0}\Lambda
\end{equation*}
holds, and on a further subsequence this convergence also holds almost surely (we restrict our attention to this subsequence without changing notation). Using the \textit{strict} monotonicity of $r\mapsto \beta(r)$, it follows that $r\mapsto a(r)$ is strictly increasing, hence invertible, so that
$$g_\varepsilon(u_\varepsilon) \xrightarrow{\varepsilon\to 0}a^{-1}(\Lambda)$$
$\vrho dx$-a.s. This yields $\varepsilon\beta(g_\varepsilon(u_\varepsilon)) \xrightarrow{\varepsilon\to 0}0$ $\vrho dx$-a.s., and hence
\begin{equation*}
u_\varepsilon = g_\varepsilon(u_\varepsilon)+ \varepsilon\beta(g_\varepsilon(u_\varepsilon)) \xrightarrow{\varepsilon\to 0}a^{-1}(\Lambda)\quad \vrho dx-\text{a.s.}.
\end{equation*}
Therefore, for $u$ from \eqref{weak-conv-to-u} we have $u = a^{-1}(\Lambda)$ a.s. and, by \eqref{Linfty-bound-u_eps}, we conclude $u \in L^\infty$ with 
\begin{align*}
|u|_\infty \leq  |f|_\infty (1  +|(\divvrho D)^- + |D| |^{\frac 1 2}_\infty),
\end{align*}
and $u_\varepsilon \xrightarrow{\varepsilon\to 0}u$ in $L^p_{\loc}$ for each $p \in [1,\infty)$. Moreover, by Fatou's lemma and the boundedness of $\{u_\varepsilon\}_{\varepsilon>0}$ in each $L^p(\vrho)$, we infer $u \in L^p(\vrho)$ for each $p \in [1,\infty]$. Furthermore, it is easily seen that $\{\tilde{\beta}_{\varepsilon}\}_{\varepsilon \in (0,1)}$ is locally equicontinuous and hence, since $\beta$ is locally Lipschitz, the convergence
\begin{equation}\label{conv-beta}
	\tilde{\beta}_{\varepsilon}(u_\varepsilon) \xrightarrow{\varepsilon\to 0}\beta(u)
\end{equation}
holds $\vrho dx$-a.s. and in $L^p_{\loc}$, $p\geq 1$. Hence $\eta = \beta(u)$.


Letting $\varepsilon \to 0$ in \eqref{approx-eq-with-A_epsilon}, it follows that $A_\varepsilon u_\varepsilon$ converges in $L^p_{\text{loc}} $, $p \in [1,\infty)$, to some limit $\Psi \in  L^1(\vrho)\cap L^\infty$. Moreover, by \eqref{conv-beta}, the following convergences hold in $\mathcal{D}'$:
\begin{equation*}
	\Delta\tilde{\beta}_{\varepsilon}(u_\varepsilon) \xrightarrow{\varepsilon\to 0}\Delta \beta(u),\quad \nabla \Phi \cdot \nabla \tilde{\beta}_{\varepsilon}(u_\varepsilon) \xrightarrow{\varepsilon\to 0 }\nabla \Phi \cdot \nabla\beta(u), \quad\varepsilon \tilde{\beta}_\varepsilon(u_\varepsilon) \xrightarrow{\varepsilon \to 0}0.
\end{equation*}
Moreover, since $(b^*_\varepsilon)_{\varepsilon >0}$ is locally equicontinuous and $D_\varepsilon \longrightarrow D$ $\vrho dx$-a.s. with $|D_\varepsilon|\leq |D|$, also 
$$\divvrho (D_\varepsilon b^*_\varepsilon(u_\varepsilon))  \xrightarrow{\varepsilon\to 0} \divvrho(Db^*(u))\text{ in }\mathcal{D'}.$$
Therefore, the identity
\begin{equation*}
\Psi=	L^1_{\textup{loc}}-\lim_{\varepsilon \to 0}A_\varepsilon u_\varepsilon = -\Delta \beta(u) + \nabla \Phi \cdot \nabla \beta(u) + \divvrho(D b^*(u))=A_0(u)
\end{equation*}
holds in $L^1(\vrho)\cap L^\infty$. In particular, $u \in D(A_0)$ and 
\begin{equation*}
	u + \lambda A_0 u = f
\end{equation*}
holds in $L^1(\vrho)$. Set
$$J_\lambda: L^1(\vrho)\cap L^\infty \to D(A_0), \quad f \mapsto J_\lambda f := u=u(\lambda,f).$$ By the $L_{\loc}^1$-convergence $u_\varepsilon \xrightarrow{\varepsilon\to 0}u$, one obtains \eqref{resolvent-eq-statement},\eqref{infty-bound-statement} and, by also using \eqref{eq:L1-contract-all-f},
\begin{equation*}\label{Jlambda-eq}
	|J_\lambda f- J_\lambda g |_{1,\vrho} \leq |f-g|_{1,\vrho}, \quad \forall f,g \in L^1(\vrho)\cap L^\infty.
\end{equation*}
Concerning \eqref{prop-measure-lemma}, we have for $f \in L^1(\vrho)\cap L^\infty$ and any pair $(\varepsilon,\lambda) \in (0,\infty)\times (0,\lambda_0)$
\begin{equation*}\label{eq:pos-preserv}
	f \geq 0 \text{ a.s.}\implies u_\varepsilon(f)  \geq 0 \text{ }\vrho dx-\text{a.s.}
\end{equation*}
Indeed, for $f \geq 0$, multiplying \eqref{approx-eq-with-A_epsilon} with $u_\varepsilon$ in place of $u$ by $\chi_\delta(u_\varepsilon^-)$, integrating with respect to $\vrho dx$, and using the monotonicity of $\tilde{\beta}_\varepsilon$ and $\chi_\delta' \geq 0$ on $[0,\infty)$ gives
\begin{equation*}
	\int_{\R^d} u_\varepsilon \chi_\delta(u_\varepsilon^-)\,\vrho dx \geq  \frac{\lambda}{\delta} \int_{\{u_\varepsilon \in [-\delta,0]\}} D_\varepsilon b_\varepsilon^*(u_\varepsilon) \cdot \nabla u_\varepsilon^-\,\vrho dx \geq -\lambda |D_\varepsilon|_\infty |b|_\infty \int_{\{u_\varepsilon \in [-\delta,0]\}}|\nabla u_\varepsilon^-|\,\vrho dx \xrightarrow{\varepsilon\to 0} 0.
\end{equation*}
Consequently, we find
\begin{equation*}
	\int_{\{u_\varepsilon \leq 0\}} u_\varepsilon\,\vrho dx = \liminf_{\varepsilon \to 0}\int_{\R^d}u_\varepsilon \chi_\delta(u_\varepsilon^-)\,\vrho dx \geq 0,
\end{equation*}
and thus also $u = \lim_{\varepsilon\to 0}u_\varepsilon \geq 0$ $\vrho dx$-a.s.
Furthermore, for non-negative $f \in L^1(\vrho)\cap L^\infty$ and $\varphi \in C^\infty_c$, Lemma \ref{lem-all-symmetry-eq-L} gives
\begin{align}\label{help-eq}
	\notag \int J_\lambda(f) \varphi \,\vrho dx &= \lim_{\varepsilon \to 0}\int_{\R^d} u_\varepsilon \varphi\,\vrho dx 
	\\&= \int f \varphi \,\vrho dx+\lambda\lim_{\varepsilon\to 0}\bigg( \int_{\R^d} \beta(u_\varepsilon)(\Delta \varphi -\ \nabla \Phi\cdot  \nabla \varphi )\,\vrho dx + \int_{\R^d}D_\varepsilon b^*_\varepsilon(u_\varepsilon)\cdot \nabla \varphi \,\vrho dx\bigg)
	\\& \notag = \int f \varphi \,\vrho dx + \lambda \int_{\R^d}\beta(J_\lambda f)\big(\Delta \varphi - \nabla \Phi\cdot \nabla \varphi\big)\,\vrho dx  + \int_{\R^d} Db^*(J_\lambda f)\cdot \nabla \varphi \,\vrho dx.
\end{align}
Since $J_\lambda(f),\beta(J_\lambda f), Db^*(J_\lambda f) \in L^\infty$, choosing $\varphi = \varphi_l$ such that $0\leq \varphi_l\leq 1$, $\varphi_l (r)= 1$ for $|r| \leq l$ and  $|\nabla \varphi_l|+|\Delta \varphi_l| \longrightarrow 0$ as $l \to \infty$ such that $|\nabla \varphi_l|+|\Delta \varphi_l|$ are bounded in $L^\infty$ uniformly in $l$ and considering \eqref{help-eq} for $l\to \infty$ gives
$$\int J_\lambda(f) \vrho dx = \int f \vrho dx,$$
where we also used $-\vrho\nabla \Phi  = \nabla \vrho \in L^1(\R^d,\R^d)$. Hence, \eqref{prop-measure-lemma} holds.
Finally, concerning \eqref{eq:density-of-L1-Linfty-in-DA}, let $g \in C^\infty_c$. For all pairs $(\varepsilon, \lambda) \in (0,\infty)\times (0,\lambda_0)$, we have $g \in D(A_\varepsilon)$ and
\begin{equation*}
	g- \lambda\Delta \tilde{\beta}_{\varepsilon}(g) +\lambda\nabla \Phi \cdot \nabla \tilde{\beta}_{\varepsilon}(g) + \lambda \varepsilon\tilde{\beta}_{\varepsilon}(g)  + \divvrho(D_\varepsilon b^*_\varepsilon(g))= g+\lambda A_\varepsilon g.
\end{equation*}
Since the right-hand side is in $L^1(\vrho)\cap L^2(\vrho)$ and $g \in L^2(\vrho)$ with $\tilde{\beta}_{\varepsilon}(g) \in D(L)$, we have
$g = u_\varepsilon(\lambda,g+\lambda A_\varepsilon g)$, which implies
\begin{equation*}
	|g-u_\varepsilon(\lambda,g)|_{L^1(\vrho)} \leq \lambda|A_\varepsilon g|_{1,\vrho}.
\end{equation*}
Since $\beta\in C^2$, it follows that $\tilde{\beta}_{\varepsilon}$, $\tilde{\beta}_{\varepsilon}'$ and $\tilde{\beta}_{\varepsilon}''$ are locally bounded uniformly in $\varepsilon \in (0,1)$, and hence
\begin{align*}
	|A_\varepsilon g|_{1,\vrho} \leq |\Delta \tilde{\beta}_{\varepsilon}(g)|_{1,\vrho}+ |\nabla \Phi \cdot \nabla \tilde{\beta}_{\varepsilon}(g)|_{1,\vrho}+\varepsilon|\tilde{\beta}_{\varepsilon}(g)|_{1,\vrho} +|\divvrho(D_\varepsilon b^*_\varepsilon(g)|_{1,\vrho}
	\leq C (|g|_{W^2,1(\vrho)}+|g|_{H^2(\vrho)}),
\end{align*}
where $C>0$ depends on $g, \beta, \Phi, b$ and $D$, but not on $\varepsilon \in (0,1)$ or $\lambda>0$. Hence by the $L^1_{\text{loc}}$-convergence of $u_\varepsilon(\lambda,g)$ to $J_\lambda g$, we conclude \eqref{eq:density-of-L1-Linfty-in-DA}, which concludes the proof of Lemma \ref{main-lemma}.
\qed
\\

Now we complete the proof of Theorem \ref{main-thm} by showing that the generalized solutions $u$ are also solutions to \eqref{equation-main} in distributional sense.
\begin{proof}[Proof of Theorem \ref{main-thm} continued]
Denote by $u=u(u_0)$ the generalized solution to \eqref{equation-main} given by Theorem \ref{main-thm}, i.e. $u(t)$ is the (locally uniform in $t$) $L^1(\vrho)$-limit as $h \to 0$ of $u_h(t)$, where $u_h$ is the step function from \eqref{eq:fin-diff-scheme} (with $A$ as in\eqref{def:op-A} instead of $\tilde{A}$), with $u_h^i = J_h^{-i}u_0$.
We have (setting $u_h(s) := u_0$ for $s \in (-\infty,0)$)
\begin{equation*}
	u_h(t) +hA(u_h(t)) = u_h(t-h),\quad t >0,
\end{equation*}
in $L^1(\vrho)$, which after multiplying with $\varphi \in C_c^\infty([0,\infty)\times \mathbb{R}^d)$ and  integrating with respect to $dt\otimes \vrho dx$ gives
\begin{align}\label{eq:approx-eq-distr-sol}
	\int_0^\infty& \int_{\mathbb{R}^d} h^{-1}(u_h(t,x)-u_h(t-h,x))\varphi(t,x)-\beta(u_h(t,x))L\varphi(t,x)\,\vrho dx dt 
	\\&\notag - \int_{\R^d} Db^*(u_h(t,x))\cdot \nabla \varphi(t,x)\,\vrho dx= 0.
\end{align}
Because the local uniform $L^1(\vrho)$-convergence in $t$ of $u_h(t)$ to $u(t)$ as $h \to 0$ implies, using \eqref{eq:exp-bound-infty-norm}, $\sup_{0<h<h_1}|u_h(t)|_\infty \leq (1+C)|u_0|_\infty$ for sufficiently small $h_1 >0$ and $C>0$ (see also Remark \ref{rem:rem} (i)), we obtain
$\beta(u_h(t)) \to \beta(u(t))$ in $L^1(\vrho)$ locally uniformly in $t$. Moreover, 
\begin{align*}
		\int_0^\infty \int_{\mathbb{R}^d} &\frac{(u_h(t,x)-u_h(t-h,x)}{h}\varphi(t,x) \,\vrho dx dt \\&= 	-\int_0^\infty \int_{\mathbb{R}^d} \frac{\varphi(t+h,x)-\varphi(t,x)}{h}u_h(t,x)\,\vrho dx dt-\int_{-h}^0\int_{\mathbb{R}^d}\frac{\varphi(t+h,x)}{h}u_0(x)\,\vrho dx dt,
\end{align*}
which, as $h\to 0$, by the locally uniform in $t$ convergence of $u_h$ to $u$ in $L^1(\vrho)$, converges to 
\begin{equation*}
	-\int_0^\infty \int_{\R^d}\partial_t\varphi(t,x)u(t)\,\vrho dx dt - \int_{\R^d} u_0\varphi(0,x)\,\vrho dx .
\end{equation*}
Consequently, letting $h\to0$ in \eqref{eq:approx-eq-distr-sol} yields
\begin{equation*}
	\int_0^\infty \int_{\R^d} \bigg(\partial_t\varphi(t,x)+\frac{\beta(u(t,x)}{u(t,x)}L \varphi(t,x) +D(x)b(u(t,x))\cdot \nabla \varphi(t,x)\,\bigg)u(t,x)\vrho dx dt +\int_{\R^d}u_0 \varphi(0,x)\,\vrho dx = 0.
\end{equation*}
Equivalently (see \cite[Thm.6.1.2]{FPKE-book15}), $v(t,x):= u(t,x)\vrho(x)$ satisfies for all $t \in (0,\infty)$
\begin{align*}
	\int_{\R^d} \varphi(t)\,v(t)dx = \int_{\R^d} &\varphi(0)\,v(0)dx 
	\\&+ \int_0^t \int_{\R^d} \frac{\beta(v(s)\vrho^{-1})}{v(s)\vrho^{-1}}\Delta \varphi + \bigg(Db(v(s)\vrho^{-1})-\frac{\beta(v(s)\vrho^{-1})}{v(s)\vrho^{-1}}\nabla \Phi \bigg)\cdot \nabla \varphi \,v(s)dx ds,
\end{align*}
i.e. $t\mapsto v(t,x)dx$ is a weakly continuous distributional solution to \eqref{equation-main} with initial datum $u_0 \vrho dx$. This completes the proof of Theorem \ref{main-thm}.
\end{proof}
Finally, we give the proof of Proposition \ref{prop:prop1}.
\\
\\
\textit{Proof of Proposition \ref{prop:prop1}.} It suffices to extend Lemma \ref{main-lemma} from $L^1(\vrho)\cap L^\infty$ to $L^1(\vrho)$. More precisely, we slightly change the definition of the domain of $A_0$ to
$$\tilde{D}(A_0) := \big\{f \in L^1(\vrho)\,|\, -L_0\beta(f)+\divvrho(Db(f)f) \in L^1(\vrho)\big\},$$
let $0< \lambda <\lambda_0$ with $\lambda_0$ as in Lemma \ref{main-lemma},
and we show: $R(I+\lambda A_0) = L^1(\vrho)$, $J_\lambda$ extends to $L^1(\vrho)$ such that
\begin{equation}\label{auxaux}
	|J_\lambda f - J_\lambda g|_{1,\vrho} \leq |f-g|_{1,\vrho}, \quad \forall f,g \in L^1(\vrho)
\end{equation}
and \eqref{resolvent-eq-statement} extends to $f \in L^1(\vrho)$. Indeed, once this is proven, all assertions of the proposition but the final one follows directly from the Crandall--Liggett nonlinear semigroup theory, this time applied on the full Banach space $L^1(\vrho)$ (compare the first part of the proof of Theorem \ref{main-thm}) in Section \ref{sect2}). That the mild solutions are also distributional solutions follows as in the final part of the proof of Theorem \ref{main-thm} by noting that here the convergence of $\beta(u_h)$ to $\beta(u)$ locally uniformly in $t$ as $h\to0$ follows from (H2').
For $f \in L^1(\vrho)$, let $(f_n)_{n \geq 1} \subseteq L^1(\vrho)\cap L^\infty$ such that $f_n \longrightarrow f$ in $L^1(\vrho)$ as $n \to \infty$. Then the contraction property of $J_\lambda$ implies the existence of an $L^1(\vrho)$-limit poin $u=u(\lambda,f)$ of $(J_\lambda f_n)_{n \geq 1}$. Since $u_n +\lambda A_0 u_n = f_n$, it follows that $(A_0 u_n)_{n \geq 1}$ has an $L^1(\vrho)$-limit point $\Psi$. By (H2'), $\beta(u_n) \longrightarrow \beta(u)$ in $L^1(\vrho)$ as $n \to \infty$, hence $L_0 \beta(u_n) \longrightarrow L_0 \beta(u)$ in $\mathcal{D}'$ as $n \to \infty$. Since, due to the boundedness of $b$ and $D$, also $Db^*(u_n) \longrightarrow Db^*(u)$ in $L^1(\vrho)$ as $n \to \infty$, we obtain $A_0u_n \longrightarrow -L_0 \beta(u) + \divvrho(Db^*(u))$ in $\mathcal{D}'$. Hence, $u \in \tilde{D}(A_0)$, $\Psi = A_0 u$, and thus $u=u(\lambda,f) \in (I+\lambda A_0)^{-1}f$. Extending $J_\lambda$ to $L^1(\vrho)$ via $J_\lambda f := u(\lambda,f)$, it is easily seen that \eqref{auxaux} holds and \eqref{resolvent-eq-statement} extends to all $f\in L^1(\vrho)$. 

Concerning the final assertion, note that $S(t)u_0 = \lim_{n \to \infty} (J_{\frac t n})^nu_0$, let $f \geq 0$, $|f|_{1,\vrho} = 1$, and set $f_n := \max(f,n) \in L^1(\vrho)\cap L^\infty$, $n\in \N_0$. We know from the proof of Theorem \ref{main-thm} that $f_n \geq 0 \implies J_\lambda f_n \geq 0$ and, for such $f_n$, $|J_\lambda f_n|_{1,\vrho} = |f_n|_{1,\vrho}$. Hence, since $J_\lambda f =L^1(\vrho)- \lim_{n \to \infty}J_\lambda f_n$, we obtain $J_\lambda f \geq 0$ and $|J_\lambda f|_{1,\vrho} = |f|_{1,\vrho}$, which implies $S(t)u_0\,\vrho dx \in \Pscr$, if $u_0 \vrho dx\in \Pscr$. \qed 

\section{Nonlinear perturbed Ornstein--Uhlenbeck processes}\label{sect:stoch-section}

Let Hypothesis 1 be satisfied.
\subsection{Existence}
In this section, we solve the McKean--Vlasov SDE \eqref{eq:SDE} with $1D$-time marginals given by the mild solutions to \eqref{equation-main}.
As said in the introduction, equation \eqref{eq:SDE} can be regarded a model for generalized nonlinear perturbed Ornstein--Uhlenbeck processes, since for $D = 0$, $\Phi = -\frac{|x|^2}{2}$ and $\beta(r)=\frac{\sigma^2}{2}r$, $\sigma >0$, one recovers the classical Ornstein--Uhlenbeck SDE
\begin{equation*}
	dX_t = -X_t dt + \sigma dB_t.
\end{equation*}
\begin{dfn}
	A \textit{(probabilistically weak) solution} to \eqref{eq:SDE} is a triple consisting of a filtered probability space $(\Omega, \Fscr, (\Fscr_t)_{t \geq 0}, \mathbb{P})$, an $\R^d$-valued $(\Fscr_t)$-Brownian motion $B$ and an $(\Fscr_t)$-adapted stochastic process $X = (X_t)_{t \geq 0}$ on $\Omega$ such that $\mathbb{P}\circ X_t^{-1} = v(t,x)dx$,
\begin{equation*}
\int_0^T D(X_t)b\big(v(t,X_t)\vrho^{-1}(X_t)\big) - \frac{\beta(v(t,X_t)\vrho^{-1}(X_t))}{v(t,X_t)\vrho^{-1}(X_t)}\nabla \Phi(X_t)dt \text{  and  } \int_0^T\frac{\beta(v(t,X_t)\vrho^{-1}(X_t))}{v(t,X_t)\vrho^{-1}(X_t)}dt
\end{equation*}
belong to $L^1(\Omega;\mathbb{P})$ for all $T>0$,
and $\mathbb{P}$-a.s.
\begin{align*}
	X_t = X_0 + \int_0^t D(X_t)&b\big(v(t,X_t)\vrho^{-1}(X_t)\big) - \frac{\beta(v(t,X_t)\vrho^{-1}(X_t))}{v(t,X_t)\vrho^{-1}(X_t)}\nabla \Phi(X_t)dt
	 \\&+ \int_0^t \sqrt{2\frac{\beta(v(t,X_t)\vrho^{-1}(X_t))}{v(t,X_t)\vrho^{-1}(X_t)}}dB_t,\quad \forall t \geq 0.
\end{align*}
\end{dfn}
Let $u_0 \in \Pscr_\infty(\vrho)$, $u=u(u_0)$ be the generalized solution to \eqref{equation-main} from Theorem \ref{main-thm} with $u(0) = u_0$ and consider the curve 
\begin{equation}\label{eqeq}
t\mapsto \nu(t,u_0):= u(t)\vrho dx,
\end{equation}
i.e. $\nu$ is a weakly continuous distributional probability solution to \eqref{equation-main} in the sense of Definition \ref{def:distr-sol}. We have
\begin{align*}
	\int_0^t \int_{\R^d} &\bigg|\frac{\beta(v(s)\vrho^{-1})}{v(s)\vrho^{-1}}\bigg| + \bigg|b(v(s)\vrho^{-1})D -\frac{\beta(v(s)\vrho^{-1})}{v(s)\vrho^{-1}} \nabla \Phi\bigg |\,v(s)dxds 
	\\& \leq \int_0^t \int_{\R^d} \big|\beta(u(s))\big|  + \big|b(u)uD\big| +\big|\beta(u(s))\nabla \Phi \big|\,\vrho dx ds<\infty,
\end{align*}
where the second inequality holds due to $u \in L_{\text{loc}}^\infty([0,\infty),L^\infty)$ (which follows from \eqref{eq:exp-bound-infty-norm}), the boundedness of $D$ and $b$, the local Lipschitz-continuity of $\beta$, and $\nabla \Phi \in L^1(\R^d,\R^d;\vrho)$. Hence the \textit{superposition principle} (see \cite[Sect.2]{BR18_2} and \cite{Trevisan16}) yields the existence of a probabilistically weak solution $(X_t)_{t \geq 0}$ to \eqref{eq:SDE} such that $\mathcal{L}(X_t) = \nu(t,u_0)$, $t\geq 0$.

If also (H2') holds, this existence result extends as follows. By Proposition \ref{prop:prop1}, there is a distributional probability solution $\nu(t,u_0):t \mapsto S(t)u_0 \,\vrho dx$ for all $u_0$ such that $u_0 \,\vrho dx \in \Pscr$. Hence, in this case, for any such $u_0$ there is a weak solution $X=X(u_0)$ to \eqref{eq:SDE} with $\mathcal{L}(X_t) = \nu(t,u_0)$, $t \geq 0$.

\subsection{Solutions to nonlinear perturbed Ornstein--Uhlenbeck equations as nonlinear Markov processes}In addition to Hypothesis 1, now we also assume
	\begin{enumerate}
		\item [(H5)] $b \in C^1(\R)$, and for each compact $K \subseteq \R$ there is $\alpha_K \in (0,\infty)$ such that
		\begin{equation*}
			|b'(r)r  +b(r)| \leq \alpha_K |\beta'(r)|,\quad \forall r \in K.
		\end{equation*}
	\end{enumerate}
	If $b \in C^1(\R)$ and $\beta'(r)>0$ for all $r \in \R$, the second part of (H5) is automatically satisfied.
	\\
	\noindent
In the previous subsection, we showed the following. For any initial datum $\nu_0 = u_0 \vrho(x) dx \in \Pscr_0$, where
\begin{equation*}
	\Pscr_0 := \{\nu \in \Pscr\,|\,\nu = u_0\vrho (x)dx, u_0 \in \Pscr_\infty(\vrho) \},
\end{equation*}
there is a weak solution $X = X(u_0)$ to \eqref{eq:SDE}  such that its curve of $1D$-time marginals is given by \eqref{eqeq}. By \eqref{eq:exp-bound-infty-norm},\eqref{sg-prop} it follows that $\nu(t,u_0) \in \Pscr_0$, $u \in L^\infty_{\text{loc}}((0,\infty),L^\infty)$, and 
\begin{equation}\label{flow-prop}
	\nu(t+s,u_0) = \nu(t,u(s)), \quad \forall t,s\geq 0, u_0 \in \Pscr_\infty(\vrho).
\end{equation}
The aim of this subsection is to prove that the path law family $\{\mathbb{P}_{u_0}\}_{u_0\in \Pscr_\infty(\vrho)}$, $\mathbb{P}_{u_0}:= \mathcal{L}(X(u_0))$ constitutes a \textit{nonlinear Markov process} in the sense of \cite{McKean1-classical,R./Rckner_NL-Markov22}. More precisely, we want to show that the following \textit{nonlinear Markov property} holds, compare \cite[Def.2.1]{R./Rckner_NL-Markov22}. Denote by $\pi_t: C([0,\infty),\R^d) \to \R^d$, $\pi_t(w) = w(t)$, the canonical projections and set $\Fscr_t := \sigma(\pi_s, 0\leq s \leq t)$, then the nonlinear Markov property is
\begin{equation*}
\mathbb{P}_{u_0}(\pi_t \in A | \Fscr_r) = p_{u_0,r,\pi_{r}}(\pi_{t-r} \in A), \quad \forall 0\leq r \leq t, u_0 \in \Pscr_\infty(\vrho), A \in \Bscr(\R^d).
\end{equation*}
Here $(p_{u_0,r,y}(dw))_{y \in \R^d}$ is a regular conditional probability kernel from $\R^d$ to $\Bscr(C([0,\infty),\R^d))$ of $\mathbb{P}_{u(r)}[\,\,\cdot \,\,| \pi_0 = y]$, $y \in \R^d$, where $u(r) = u(u_0)(r)$.
First, we observe
$$\nu_0 \in \Pscr_0, \tilde{\nu}_0 \in \Pscr, \tilde{\nu}\leq C \nu_0 \text{ for some }C\geq 1 \implies \tilde{\nu}_0 \in \Pscr_0.$$
Therefore, by \eqref{flow-prop} and \cite[Thm.3.7, Cor.3.8]{R./Rckner_NL-Markov22}, it suffices to prove that the \textit{linearized Fokker--Planck equations}, obtained by fixing a priori the densities $t\mapsto u(t)\vrho$ in the measure component of the coefficients in the distributional formulation of \eqref{equation-main}, are well-posed in a suitable subclass, see Lemma \eqref{lem:lin-u} after the following definition.

\begin{dfn}
	Let $s \geq 0$ and $y \in L^\infty_{\text{loc}}((s,\infty),L^\infty)$. A weakly continuous curve $\zeta: [s,\infty)\to \Pscr$ is a \textit{(distributional) probability solution to the $y\vrho$-linearized version of \eqref{equation-main}} with initial datum $\zeta_s$, if $\int_s^T \int_{\R^d} |\nabla \Phi| \,d\zeta_t dt <\infty$ for all $T\geq s$, and
	\begin{equation}\label{eqaux1}
		\int_{\R^d} \varphi \,d\zeta_t = \int_{\R^d} \varphi \,d\zeta_s + \int_s^t \int_{\R^d} \frac{\beta(y)}{y}L\varphi +b(y)D \cdot \nabla \varphi\,d\zeta_s ds,\quad \forall t \geq s, \varphi \in C^\infty_c.
	\end{equation}
\end{dfn}
Equation \eqref{eqaux1} is obtained by a priori replacing $v(s)$ in \eqref{eqaux0} by the fixed curve $s\mapsto y(s)\vrho$.
Note that since $b$ and $D$ are bounded, and since $r \mapsto \frac{\beta(r)}{r}$ is locally bounded and $y \in L^\infty_{\text{loc}}((s,\infty),L^\infty)$, the local-global integrability condition for all coefficients except for $\nabla \Phi$ is automatically satisfied, i.e.
\begin{equation*}
	\int_s^T \int_{\R^d} \big|\frac{\beta(y)}{y}\big| + \big|b(y)D\big| \,d\zeta_t dt < \infty, \quad \forall T\geq s,
\end{equation*}
for any measurable curve $t \mapsto \zeta_t \in \Pscr$.

\begin{lem}\label{lem:lin-u}
	Let $\nu_0 \in \Pscr_0$, $\nu_0 = u_0\vrho dx$, and $\nu(t,u_0) = u(t)\vrho dx$ be as in \eqref{eqeq}. For each $s \geq 0$, $t\mapsto \nu(t,u_0)$, $t \geq s$, is the unique probability solution in $L^\infty_{\textup{loc}}((s,\infty),L^\infty)$ to the $u\vrho$-linearized version of \eqref{equation-main} in the sense of the previous definition with initial datum $\nu(s,u_0)$.
\end{lem}
\begin{proof}
The assertion follows analogously to the proof of \cite[Thm.4.1]{BR22}. Indeed, under our assumptions (H1)-(H5), the proof of that result, as well as the proof of its nonlinear version, Theorem 3.2. of the same reference, can be repeated for the operator $L$ and the spaces $H^k(\vrho)$ instead of $\Delta$ and $H^k$, $k \in \N_0$, respectively. Indeed, since our assertion is restricted to probability solutions, we prove uniqueness of solutions in $L^\infty_{\text{loc}}((s,\infty),L^\infty\cap L^1)$, while in \cite{BR22} uniqueness is even proven in the larger class of solutions in $L^\infty_{\text{loc}}((0,\infty),H^{-1})$.
\end{proof}
Finally, we obtain the desired result.
\begin{prop}
	There is a nonlinear Markov process $\{\mathbb{P}_{u_0}\}_{u_0 \in \Pscr_\infty(\vrho)}$, which consists of path laws of solutions to \eqref{eq:SDE} and whose $1D$-time marginals are given by $u(t)\vrho dx$, where $u(t) = S(t)$, $t \geq 0$, is the flow of solutions to \eqref{equation-main} constructed in Theorem \ref{main-thm}. Moreover, for each $u_0\in \Pscr_\infty(\vrho)$, $\mathbb{P}_{u_0}$ is the only solution law to \eqref{eq:SDE} with initial datum $u_0 \vrho dx$ and $1D$-marginals $\nu(t,u_0)$, $t \geq 0$.
\end{prop}
\begin{proof}
	The result follows from \ref{flow-prop}, Lemma \ref{lem:lin-u} and \cite[Cor.3.8]{R./Rckner_NL-Markov22}.
\end{proof}

\paragraph{Acknowledgements.}Funded by the German Research Foundation (DFG) - Project number 517982119. The author would like to thank Michael Röckner for valuable discussions, as well as the research group of Franco Flandoli in Pisa for their hospitality.

\bibliographystyle{plain}
\bibliography{bib-collection}
\end{document}